\documentclass[12pt]{amsart}

\usepackage{amsfonts}
\usepackage{amssymb}
\usepackage{graphicx}
\usepackage{amsmath}
\usepackage{latexsym}
\usepackage{amscd}
\usepackage{xypic}
\usepackage[all,cmtip]{xy}
\usepackage{mathrsfs}%scrfont}
\usepackage{enumitem} %http://ctan.org/pkg/enumitem
\usepackage{braket}
\usepackage[margin=1in]{geometry}
\usepackage{esint}
\usepackage{dsfont}
\usepackage{pgf}
\usepackage{xypic}

\usepackage{hyperref}

\allowdisplaybreaks

\numberwithin{equation}{section}

\newtheorem{prop}{Proposition}[section]
\newtheorem{theo}[prop]{Theorem}
\newtheorem{lemm}[prop]{Lemma}
\newtheorem{coro}[prop]{Corollary}

\theoremstyle{definition}

\newtheorem{rema}[prop]{Remark}
\newtheorem{exam}[prop]{Example}

\newcommand{\CC}{\mathbb{C}}

\newcommand{\RR}{\mathbb{R}}
\renewcommand{\SS}{\mathbb{S}}

\newcommand{\sH}{\mathscr{H}}

\DeclareMathOperator{\Span}{span}

\DeclareMathOperator{\supp}{supp}

\newcommand{\bangle}[1]{\left\langle #1 \right\rangle}

\define{\sff}{h}
\define{\tfsff}{\accentset{\circ}{\sff}}

\let\oldmarginpar\marginpar
\renewcommand\marginpar[1]{\-\oldmarginpar[\raggedleft\footnotesize #1]%
{\raggedright\footnotesize #1}}

\DeclareMathOperator{\re}{Re}
\DeclareMathOperator{\imag}{Im}

\DeclareMathOperator{\Index}{index}

\DeclareMathOperator{\Res}{Res}

\usepackage{graphicx}

\allowdisplaybreaks

\title{On the topology and index of minimal surfaces II}

\author{Otis Chodosh}
\address{Department of Mathematics, Bldg. 380, Stanford University, Stanford, CA 94305, USA}
\email{ochodosh@stanford.edu}

\author{Davi Maximo}
\address{Department of Mathematics, University of Pennsylvania, Philadelphia,
PA 19104, USA}
\email{dmaxim@math.upenn.edu}

\date{\today}

\begin{document}
	\begin{abstract}
For an immersed minimal surface in $\RR^{3}$, we show that there exists a lower bound on its Morse index that depends on the genus and number of ends, counting multiplicity. This improves, in several ways, an estimate we previously obtained bounding the genus and number of ends by the index. 

Our new estimate resolves several conjectures made by J.\ Choe and D.\ Hoffman concerning the classification of low-index minimal surfaces: we show that there is no complete two-sided immersed minimal surface in $\RR^{3}$ of index two, complete embedded minimal surface with index three, or complete one-sided minimal immersion with index one. 
\end{abstract}

\maketitle

\section{Introduction}

Let $X:\Sigma\rightarrow \mathbb{R}^3$ be a complete minimal immersion. We assume throughout  $\Sigma$ has finite Morse index. By work of Fischer-Colbrie and Gulliver--Lawson \cite{Fischer-Colbrie:1985,Gulliver-Lawson,Gulliver:indexFTC} in the two-sided case, and Ross \cite{Ross} and Ros \cite{Ros:one-sided} in the one-sided case, this is equivalent to $\Sigma$ having finite total curvature. Thus, such a $\Sigma$ is conformally equivalent to a closed Riemann surface $\overline\Sigma$ of finite genus $g$ with finitely many punctures $p_1,p_2,\ldots, p_r$, which correspond to the ends of $\Sigma$. We refer the reader to the introduction of \cite{CM16} for a thorough overview of existing results concerning the index of complete minimal surfaces in $\RR^{3}$. 

In our previous work \cite{CM16}, we established the following lower bound for the index of two-sided surfaces:
\begin{equation}\label{eq:cm16}
\Index(\Sigma) \geq \frac{2}{3}\left(g+r\right)-1.
\end{equation}
We used this to prove the non-existence of embedded minimal surfaces of index two. However, \eqref{eq:cm16} fails to account for the multiplicity of the end and thus is not very useful for the study of immersed minimal surfaces.\footnote{For example, applying \cite[Corollary 15]{MontielRos} to the $k$th-order Enneper's surface $X_{k}:\CC\to\RR$ (the minimal surface with Weierstrass data $g(z) = z^{k},dh=z^{k}dz$ in the notation of \cite{HoffmanKarcher}), we find that $\Index X_{k} = 2k-1\to\infty$ as $k\to\infty$. On the other hand, the right hand side of \eqref{eq:cm16} is $-\frac 13$ for all $k$. Note that Theorem \ref{theo:main} yields the lower bound $\frac{1}{3}(2k+1)$.}

In this work, we substantially improve our estimate \eqref{eq:cm16} both in terms of the numerical constants and also so as to account for the multiplicity of the ends. As a consequence, we obtain several new classification results for complete minimal surfaces of low index. Our main results are as follows.

\begin{theo}\label{theo:main}
Let $X:\Sigma\rightarrow\mathbb{R}^3$ be a two-sided complete immersed minimal surface. Suppose $\Sigma$ has genus $g$ and $r$ ends $E_1,E_2,\ldots, E_r$, with multiplicities respectively $d_1,d_2,\ldots,d_r$. Then 
\[
\Index(\Sigma) \geq \frac{1}{3}\left(2g+2\sum^{r}_{j=1}(d_j+1)-5\right).
\]
\end{theo}

This is stronger than our bound \eqref{eq:cm16} from \cite{CM16} even in the case when there are no multiplicities (i.e., the ends are embedded). Indeed, we have: 

\begin{coro}\label{cor:emb}
Let $X:\Sigma\rightarrow\mathbb{R}^3$ be a two-sided complete immersed minimal surface of genus $g$ and with $r$ embedded ends. Then,
\[
\Index(\Sigma) \geq \frac{1}{3}\left(2g+4r-5\right).
\]
\end{coro}

Our method also extends to the one-sided case, in which case we find
\begin{theo}\label{theo:nonor}
Let $\Sigma$ be a one-sided complete immersed minimal surface in $\mathbb{R}^3$ of finite total curvature. Suppose that the ends of $\Sigma$, $E_1,E_2,\ldots, E_r$, have multiplicities $d_{1},\dots,d_{r}$ and that the genus of the two-sided double cover $\hat \Sigma$ is $g$. Then,
\[
\Index(\Sigma) \geq \frac{1}{3}\left(g+2\sum^{r}_{j=1}(d_j+1)-4\right).
\]
\end{theo}

Our proof of Theorems \ref{theo:main} and \ref{theo:nonor} proceed as in \cite{CM16} but with a more delicate choice of weighted space (see Section \ref{sec:weights}). Our new weighted spaces allow us to consider harmonic $1$-forms in the stability inequality with even slower decay towards the ends of the minimal surface than those considered previously. We note that the construction of the index in this space (in the sense of Fischer-Colbrie \cite{Fischer-Colbrie:1985}) in Section \ref{sec:weightedFC} is somewhat subtle and requires a careful choice of cutoff function.

\begin{exam}\label{exam:cat}
For example, on the catenoid $\Sigma = \CC\setminus\{0\}$, the weighted spaces used in \cite{CM16} allowed us to plug in (the real and imaginary components of) $\frac{dz}{z}$. Here, our new weight allows the additional forms $dz$ and $\frac{dz}{z^{2}}$. This yields the \emph{sharp}\footnote{To be precise, the method from \cite{CM16} would yield $\Index(\Sigma) \geq \frac 1 3$ when $\Sigma$ is the catenoid (which has index one). This, of course, implies the ``sharp bound'' $\Index(\Sigma) \geq 1$ after taking the integer part. The point is that our current estimate is \emph{saturated} (even before taking the integer part) in the case of the catenoid, so one should not hope to improve the method even further by finding more forms corresponding to the ends. } lower bound\footnote{We have $6$ linearly independent harmonic $1$-forms to start (the real and complex parts of the given $3$ holomorphic forms). We subtract $3$ to remove the ``linear'' forms $*dx^{1},*dx^{2},*dx^{3}$. Finally, we must divide by $3$ when using the method developed by Ros \cite{Ros:one-sided} to estimate the index (see Section \ref{subsec:outline-method}).} of $\Index(\Sigma)\geq 1$. 
\end{exam}

\subsection{Applications} 

As noted above, our improved bounds yield new classification results for low-index minimal surfaces. In \cite{CM16}, we proved there are no complete embedded minimal surfaces in $\RR^{3}$ of index two. Here, can improve this in two ways: in Theorem \ref{theo:index2} we extend this to rule out immersed index two surfaces and in Theorem \ref{theo:index3} we extend this to rule out embedded index three surfaces. We also show that there are no index one one-sided minimal immersions in Theorem \ref{theo:index-1-nonor}. 

The following was conjectured by Choe \cite[p.\ 210]{Choe:vision}. 
\begin{theo}\label{theo:index2}
There is no complete two-sided minimal immersion into $\RR^{3}$ with index two. 
\end{theo}
\begin{proof}
Consider $X:\Sigma\to\RR^{3}$ of index two. By \cite[Theorem 7]{Choe:vision} or \cite[Corollary 3.3]{Nayatani}, we can assume that $\Sigma$ has genus $g\geq 1$. Applying Theorem \ref{theo:main}, we find that 
\[
2g+2\sum^{r}_{j=1}(d_j+1) \leq 11. 
\]
We prove in Lemma \ref{lemm:tot-wind-bds} below that any non-flat immersed minimal surface of finite total curvature has
\[
\sum^{r}_{j=1}(d_j+1)\geq 4.
\]
Thus, $g = 1$ and $\sum^{r}_{j=1}(d_j+1) = 4$.  
Combined with the Jorge--Meeks formula \eqref{eq:jm}, we find
\[
\int_{\Sigma} \kappa =  -8\pi,
\]
where $\kappa$ is the Gauss curvature of $\Sigma$. By work of L\'opez \cite{Lopez:12pi}, because $g=1$, the only possibility is that $X:\Sigma\to\RR^{3}$ is the Chen--Gackstatter surface (cf.\ \cite[\S 2.2]{HoffmanKarcher}). However, the Chen--Gackstatter surface has index $3$ by work of Montiel--Ros \cite[Corollary 15]{MontielRos}. This concludes the proof. 
\end{proof}

\begin{rema}
We note that a similar argument can be used to give a new proof of the classification of complete immersed minimal surfaces of index one by L\'opez--Ros \cite{LopezRos:index-one} (they must be either Enneper's surface or the catenoid). 
\end{rema}

The following was posed to us by Hoffman (see \cite[p.\ 402]{CM16}). 

\begin{theo}\label{theo:index3}
There is no complete embedded minimal surface in $\RR^{3}$ with index three. 
\end{theo}
\begin{proof}
Consider an embedded minimal surface $\Sigma\subset \RR^{3}$ with $\Index(\Sigma) = 3$. Using Theorem \ref{theo:main}, we find
\[
 g+2r \leq 7.
\]
Since the ends are embedded, the only examples with $r\leq 2$ are the plane and catenoid by \cite{Schoen:symmetry}, with index $0$ and $1$ respectively. Thus $r \geq 3$. Similarly, by work of L\'opez--Ros \cite{LopezRos:genus-zero} the plane and the catenoid are the only embedded genus zero examples, so we can assume that $g\geq 1$. 

Combining the above inequality with $r\geq 3$ and $g\geq 1$, we find that $g =1$ and $r=3$. Thus by Costa's characterization of embedded minimal tori with $3$-ends \cite{Costa:JDG,Costa:Invent}, $\Sigma$ must be a member of the deformation family of the Costa surface, as constructed by Hoffman--Meeks (see \cite{Costa:Invent} and \cite[\S 4]{HoffmanKarcher}). 

We show in Section \ref{sec:costa} that the three ended tori in the in deformation family have index at least $4$.\footnote{Nayatani has shown that the index of the Costa surface itself is exactly $5$ \cite{Nayatani:indexGauss} (see also \cite{Nayatani:CHM-index,Morabito}). However, it seems to be unknown how the index behaves along the deformation family \cite{Nayatani:comm,Morabito:comm}. Choe's vision number argument \cite[Corollary 5]{Choe:vision} applies in this case to prove that it is at least $3$, but this does not suffice for us here (alternatively the lower bound of $3$ for the index along the deformation family follows from Theorem \ref{theo:main}). } This is a contradiction. 
\end{proof}

We also confirm the following conjecture of Choe \cite[p.\ 210]{Choe:vision}:

\begin{theo}\label{theo:index-1-nonor}
There is no complete one-sided minimal immersion in $\RR^{3}$ with index one.
\end{theo}
\begin{proof}
Consider $X:\Sigma\to\RR^{3}$ a complete one-sided minimal immersion with index one. Theorem \ref{theo:nonor} implies that
\[
g+2\sum^{r}_{j=1}(d_j+1) \leq 7.
\]
Because $\sum_{j=1}^{r}(d_{j}+1)\geq 4$ by Lemma \ref{lemm:tot-wind-bds}, this is a contradiction.
%Consider $X:\Sigma\to\RR^{3}$ a non-orientable surface of index one. Since $d_j\geq1$, Theorem \ref{theo:nonor} implies that $7\geq 4r$ and thus $r=1$. We claim $d_1\geq2$. If not, the monotonicity formula applied to $X(\Sigma)$ implies that $X(\Sigma)$ is a flat plane. This is a contradiction. 
%
%On the other hand, Theorem \ref{theo:nonor} implies that $d_1\leq2$, so $d_1=2$.  By the same inequality, $g\leq1$. We then estimate the total curvature of $\Sigma$ by the Jorge--Meeks formula \eqref{eq:jmnon} to be at least $-6\pi$. Using a result of Meeks \cite[Corollary 2]{Meeks81}, $\Sigma$ must be conformal to a punctured projective plane, which by Choe \cite[Theorem 6]{Choe:vision} has index at least two.
\end{proof}

\begin{rema}
Ambrozio--Buzano--Carlotto--Sharp have observed \cite{ABCS:closed-compactness} that our previous work \cite{CM16} (along with L\'opez--Ros \cite{LopezRos:index-one} and our work\footnote{See also the subsequent work of Buzano--Sharp \cite{buzanoSharp}.} with Ketover \cite{CKM} on the degeneration of bounded index minimal surfaces) leads to compactness results for minimal surfaces with low index but comparably large genus in (non-bumpy) Riemannian $3$-manifolds with positive scalar curvature. Our Theorems \ref{theo:main} and \ref{theo:index3} can be used to sharpen some of their results. 

We also note that the same authors have observed that our work \cite{CM16} is also relevant to the study of boundary degeneration of index one free boundary minimal surfaces \cite{ABCS:fb}. The analysis of the index of free boundary surfaces in a half-space has some features in common with our arguments in Section \ref{sec:costa} concerning the index of the deformation family.  
\end{rema}

Finally, we explain how our main results imply a qualitative version of the fact that finite index is equivalent to finite total curvature. Theorems \ref{theo:main} and \ref{theo:nonor}, combined with the Jorge--Meeks formula yield the following result (the upper bound is due\footnote{The same bound with a worse constant was originally proven by Tysk \cite{Tysk}. See also \cite{Nayatani:indexGauss,GriYau03}.} to Ejiri--Micallef \cite{EjiriMicallef}).

\begin{theo}
If $X:\Sigma\to\RR^{3}$ is a non-planar two-sided minimal immersion, then
\[
\frac 1 3 + \frac{1}{6\pi} \int_{\Sigma} (- \kappa) \leq \Index(\Sigma) \leq  -3 + \frac{3}{2\pi} \int_{\Sigma} (-\kappa).
\]
For $X:\Sigma\to\RR^{3}$ a non-planar one-sided minimal immersion, we have
\[
\frac 1 3 + \frac{1}{6\pi} \int_{\Sigma} (-\kappa) \leq \Index(\Sigma) \leq  -6 + \frac{3}{\pi} \int_{\Sigma} (-\kappa).
\]
\end{theo}
That such a bound should hold was originally conjectured by Fischer-Colbrie \cite{Fischer-Colbrie:1985} and more explicitly by Grigor'yan--Netrusov--Yau \cite{GriYau03,GNY:eig}. It can be seen as a quantitative version of the fact that finite total curvature is equivalent to finite index \cite{Fischer-Colbrie:1985,Gulliver-Lawson,Gulliver:indexFTC}. We emphasize that by combining \cite[Theorem 4]{Nayatani:indexGauss} with \cite[Theorem 17]{Ros:one-sided} one can give a different proof of the lower bound in the case of two-sided surfaces, albeit with a much worse constant. 
\begin{proof}
If $X:\Sigma\to\RR^{3}$ is a two-sided minimal immersion, then Theorem \ref{theo:main} yields
\begin{align*}
\Index(\Sigma) & \geq \frac{1}{3}\left(2g+2\sum^{r}_{j=1}(d_j+1)-5\right)\\
%& = \frac{1}{3}\left(2g-2+\sum^{r}_{j=1}(d_j+1)\right) + \frac 1 3 \sum_{j=1}^{r}(d_{j}+1) - 1 \\
& = \frac{2}{3}\left(g-1+\frac 12 \sum^{r}_{j=1}(d_j+1)\right) + \frac 1 3 \sum^{r}_{j=1}(d_j+1) - 1\\
& = - \frac{1}{6\pi} \int_{\Sigma}\kappa + \frac 1 3 \sum^{r}_{j=1}(d_j+1) - 1.
\end{align*}
The final equality follows from the Jorge--Meeks formula \eqref{eq:jm}. Since $\sum_{j=1}^{r}(d_{j}+1)\geq 4$ for a non-flat minimal surface (cf.\ \cite[Proposition 3.1]{HoffmanKarcher}), we find 
\[
\Index(\Sigma) \geq \frac{1}{3}-\frac{1}{6\pi}\int_{\Sigma}\kappa,
\]
as claimed. 

On the other hand, \cite[Theorem 3.2]{EjiriMicallef} yields
\begin{align*}
\Index(\Sigma) & \leq -\frac1\pi \int_{\Sigma}\kappa + 2 g - 3 \\
& =  - \frac{3}{2\pi}\int_{\Sigma} \kappa - \sum_{j=1}^{r}(d_{j}+1) - 1\\
& \leq - \frac{3}{2\pi}\int_{\Sigma} \kappa - 3
\end{align*}
In the second step, we have used the Jorge--Meeks formula \eqref{eq:jm}. 

The argument in the one-sided case uses a similar argument based on Theorem \ref{theo:nonor} for the lower bound. For the upper bound, we combine \cite[Theorem 3.2]{EjiriMicallef}  with the fact that the index of the two-sided double cover is at least the index of the one-sided surface (and the total curvature is doubled). 
\end{proof}

\subsection{Outline of the method} \label{subsec:outline-method} We briefly indicate the strategy used to prove Theorem \ref{theo:main}. The heart of our approach goes back to Ros \cite{Ros:one-sided} who discovered an ingenious method to relate the topology of a minimal immersion into $\RR^3$ to the index via harmonic $1$-forms. Namely, Ros observed that if $\omega$ is a harmonic $1$-form on a two-sided minimal surface $\Sigma$ in $\RR^3$ then formally:
\[
\sum_{i=1}^3 Q(\bangle{\omega,dx_i},\bangle{\omega,dx_i}) = 0
\]
where $Q$ is the second variation operator (see Section \ref{subsec:two-sided-immersions}). This essentially follows by comparing the Bochner formula for harmonic $1$-forms on a surface to the form of the second variation operator, but this is not completely precise, as we may only plug in compactly supported functions to the second variation form. We will ignore this issue for now in this sketch, but it turns out to be a completely central issue, and we will return to it below. 

As such, if $\bangle{\omega,dx_i}$ is orthogonal to the index (negative eigenspace of $Q$) for $i=1,2,3$, then such an equality implies that $\bangle{\omega,dx_i}$ must be nullity for $Q$, i.e.\ it must be a Jacobi field. However, Ros proved that this is only possible for specific $\omega$, namely, only for forms in the span of $*dx_1,*dx_2,*dx_3$ (which correspond to the translational Jacobi fields). As such, we see that if there are many more harmonic $1$-forms than the index, then we can find $\omega$ that is not in the span of $*dx_1,*dx_2,*dx_3$ and so that $\bangle{\omega,dx_i}$ is orthogonal to the index for $i=1,2,3$. However, this is a contradiction. Thus, we see that there is a bound on the number of allowable harmonic $1$-forms in terms of the index. Considering the linear algebra in the statement we have just given, we arrive at the bound
\[
\frac 13 ( h_A - 3 ) \leq \Index(\Sigma),
\]
where $h_A$ is the dimension of the allowable harmonic $1$-forms. Indeed, we work in a co-dimension $3$ subspace of the allowable harmonic $1$-forms so as to avoid the span of $  *dx_1,*dx_2,   \allowbreak*dx_3$. Then, the orthogonality requirement boils down to $3
\times \Index(\Sigma)$ equations (since we have to arrange orthogonality for $i=1,2,3$). We know that these equations cannot have a non-zero solution by the above discussion, so $3\Index(\Sigma) \geq h_A - 3$. 

As such, the key question is how to relate $h_A$ to the underlying geometry/topology. Essentially, this boils down to the question of ``which harmonic $1$-forms can actually be plugged into the second variation via a cutoff argument?''

In \cite{Ros:one-sided}, it is shown that the $L^2(\Sigma)$-harmonic forms can be plugged into the second variation. This leads to $h_A = 2g$ (for $g$ the genus of $\Sigma$), essentially since the $L^2$ norm of a $1$-form is conformally invariant and a $L^2$-harmonic form extends across a point singularity. Our previous paper \cite{CM16} built on this by observing that it is possible to plug in more $1$-forms by considering a weighted $L^2$-space. This paper pushes this idea further by finding a weighted space that is ``optimal'' in some sense (cf.\ Example \ref{exam:cat}, Remark \ref{rema:opt}). 

We now explain the issues involved with the weighted $L^2$-spaces. A basic observation concerning the Morse index is that if one uses a weighted $L^2$-space (denoted in this paper $L^2_*$), since the index is the limit of the index of compact sets, the index is insensitive to the weight ($L^2_*$ and $L^2$ are equivalent norms for functions supported in a fixed compact set). However, the advantage of considering a weighted space is that one can prove a weighted version of Fischer--Colbrie's result (see \cite{Fischer-Colbrie:1985}) concerning an $L^2$- (here $L^2_*$) basis of the index. This makes it possible for the integration by parts arguments necessary to plug in harmonic $1$-forms that lie in $L^2_*$. 

The tension in this method is that one would like to choose a weight (to define $L^2_*$) so that the maximal number of harmonic $1$-forms are in $L^2_*$. On the other hand, the weight must be sufficiently controlled to be able to prove a weighted Fischer--Colbrie result (see Proposition \ref{prop:weighted-FC}) as well as carry out the cut-off function argument to plug $\bangle{\omega,dx_i}$ into the second-variation (see Section \ref{sec:proof-main}). The key idea here is to design $L^2_*$ so as to be compatible with the logarithmic cutoff function (cf. Remark \ref{rema:opt}).

Finally, we remark that to prove the non-existence of embedded index $3$ surfaces, we have to modify this method to study the Morse index of the Hoffman--Meeks deformation family of the Costa surface (see Section \ref{sec:costa}). Because the deformation family has two reflection symmetries, we study the problem as it decomposes into even/odd functions under such symmetries. This allows us to slightly improve our estimates (since we can compute the parity of the allowable harmonic $1$-forms as well as estimate the parity of the index). This slight improvement is enough for us to prove that the deformation family has index at least $4$. 

\subsection{Acknowledgements} We are grateful to David Hoffman, Filippo Morabito, and Shin Nayatani for valuable discussions related to the index of the deformation family. We are grateful to Ivaldo Nunes for pointing out a mistake in the proof of Proposition \ref{prop:weighted-FC} in the first version of this paper as well as to Shuli Chen for a careful reading of a draft of the paper. We are also grateful to the referee for several helpful comments. O.C.\ has been supported in part by the Oswald Veblen Fund, the NSF grants DMS-1811059 and DMS-1638352, as well as by Terman and Sloan Fellowships. The majority of this work was completed during time he was a post-doc at Princeton University and the Institute for Advanced Study. D.M.\ was supported in part by an NSF grant DMS-1737006, DMS-1910496, and a Sloan Fellowship.

\subsection{Organization} 

In Section \ref{sec:defi} we give the relevant definitions and background, as well as recalling several useful computations. In Section \ref{sec:weights} we describe our new weighted space and study the harmonic forms lying in this space. In Section \ref{sec:weightedFC}, we construct weighted eigenfunctions realizing the index. In Section \ref{sec:proof-main} we prove the main theorem, Theorem \ref{theo:main} and in Section \ref{sec:nonor} we prove the non-orientable version, Theorem \ref{theo:nonor}. Finally, in Section \ref{sec:costa} we study the index of the deformation family containing the Costa surface.

\section{Finite Morse index immersions}\label{sec:defi}

Throughout this note, we shall always assume that $X:\Sigma\rightarrow\mathbb{R}^3$ is a complete immersed minimal surface in $\RR^{3}$ of finite index.  Apart from Sections \ref{sec:non} and \ref{sec:nonor}, we shall always assume that $\Sigma$ is two-sided. We will denote by $|\cdot|$ the Euclidean distance and by $B_{R} = \{x\in\RR^{3} : |x|< R\}$ the extrinsic ball of radius $R$. Moreover, $C$ will denote a positive constant, allowed to change from line to line. 

\subsection{Two-sided immersions}\label{subsec:two-sided-immersions}For $X:\Sigma\rightarrow\mathbb{R}^3$ two-sided minimal immersion in $\RR^{3}$, we consider the stability operator defined by $L:=-\Delta+2\kappa$, where $\Delta$ is the intrinsic Laplacian associated to $X^{*}(g_{\RR^{3}})$, and the associated quadratic form
\begin{equation*}
Q(\phi,\phi) : = \int_{\Sigma}|\nabla \phi|^{2}+2\kappa\phi^{2}.
\end{equation*}
Here, $\kappa$ is the Gauss curvature of $\Sigma$.  We note that $\Sigma$ is of finite index if and only if it has finite total curvature \cite{Fischer-Colbrie:1985} and, hence, it is conformally equivalent to a compact Riemann surface $\overline \Sigma$, punctured at finitely many points $p_{1},\dots,p_{r}$ \cite{Osserman:FTC}. Moreover, the Gauss map extends across the punctures as a meromorphic map and, in particular, such a $\Sigma$ is properly immersed and we may also compute the index by taking the limit of extrinsic balls:
\begin{equation*}
\Index (\Sigma) = \lim_{R\to\infty} \Index(\Sigma \cap B_{R}).
\end{equation*}
where $\Index(\Sigma \cap B_{R})$ is the number of negative eigenvalues of $L$ in $\Sigma\cap B_R(0)$ with Dirichlet boundary condition. We refer to \cite{CM16} and the references therein for details. 

For a each $j=1,2,\ldots, p_r$, we consider $D_j$ a small punctured disk neighborhood of $p_j$ in $\overline \Sigma$. We refer to $E_j=X(D_j)$ as an end of $\Sigma$. Following \cite{JorgeMeeks}, for a fixed $R>0$, we consider $\Gamma_{R,j}=\{q \in \mathbb{S}^2|\,Rq\in E_j\}$, where $\mathbb{S}^2$ is the unit sphere in $\mathbb{R}^3$, and observe that $\Gamma_{R,j}$ converges smoothly as $R\rightarrow\infty$ to a great circle of $\mathbb S^2$ with finite multiplicity, which we refer to as the multiplicity of $E_j$ and denote by $d_j$. Moreover, Jorge-Meeks \cite{JorgeMeeks} showed that:
\begin{equation}\label{eq:jm}
\deg(N) = -\frac{1}{4\pi} \int_\Sigma \kappa = g-1 + \frac{1}{2}\left(r+\sum_{j=1}^{r} d_j\right)
\end{equation}
where $N : \overline \Sigma \to \SS^2$ is the (extended) Gauss map. 

\subsection{One-sided immersions}\label{sec:non} For an one-sided immersions $X:\Sigma\rightarrow\mathbb{R}^3$, the Morse index and total curvature can be defined using  the two-sheeted orientable covering:
$$\pi:\widehat\Sigma \rightarrow \Sigma.$$ 

For instance, if $\tau:\widehat{\Sigma} \rightarrow \widehat \Sigma$ denotes the change of sheets involution and $N$ is the unit normal vector of $\widehat{\Sigma}$, then $\tau\circ N = -N$ and an infinitesimal  compactly supported deformation of $\Sigma$ correspond to a function $\phi$ with compact support in $\widehat{\Sigma}$ satisfying $\phi \circ \tau = -\phi$. Thus, if $Q$ is the quadratic form associated with the second variation of area for $\widehat \Sigma$, the index of $\Sigma$ may be defined as the number of negative eigenvalues of $Q$ with respect to deformations that are anti-invariant with respect to $\tau$. 

The total curvature of $\Sigma$, on the other hand, is equal to the degree of its ``Gauss map" $n:\Sigma\rightarrow \mathbb{RP}^2 $, which in turn can be defined via the following diagram:
\[
\xymatrix{
\widehat\Sigma\ar[r]^N \ar[d]^\pi & \mathbb{S}^2 \ar[d]^p  \\
\Sigma \ar[r]^n & \mathbb{RP}^2\\
}
\]
where $p:\mathbb{S}^2\rightarrow\mathbb{RP}^2$ is the standard covering of $\mathbb {RP}^2$. It follows from Ros \cite[Theorem 17]{Ros:one-sided} that $\Sigma$ has finite Morse index if and only if it has finite total curvature. The total curvature of $\Sigma$ can then be estimated by applying the Jorge-Meeks formula to $\widehat\Sigma$. Indeed, if $\widehat \Sigma$ has genus $g$ and $s$ ends, then $s$ must be even, say $s=2r$ and the ends of $\widehat{\Sigma}$ must be of the form $E_1,E_2,\ldots, E_r,\allowbreak \tau(E_1), \tau(E_2),\ldots, \tau(E_r)$, and so:
\begin{align*}
\deg (n) &= \deg(N) \\
               & = g-1 + \frac{1}{2}\left(s+\sum_{j=1}^{s} d_j\right).
\end{align*}
Finally, since $E_j$ and $\tau(E_j)$ must have the same multiplicity, the above reduces to:
\begin{equation}\label{eq:jmnon}
-\frac{1}{2\pi}\int_{\Sigma}\kappa = \deg (n) = g-1 + \left(r+\sum_{j=1}^{r} d_j\right).
\end{equation}

\subsection{The total winding number}

We include the proof of the following well-known but useful estimate. We are grateful to the referee for pointing out a simplification in the two-sided case. 
\begin{lemm}\label{lemm:tot-wind-bds}
Suppose that $X :\Sigma \to\RR^3$ is a complete minimal immersion with finite total curvature. If $X(\Sigma)$ is not a flat plane, then
\[
\sum_{j=1}^r (d_j + 1) \geq 4
\]
where $d_1,\dots,d_r$ are the multiplicities of the ends of $X$. 
\end{lemm}
\begin{proof}
Suppose that $X$ is a non-flat two-sided immersion. Note that $d_j\geq 1$ for $j=1,\dots,r$ by definition. As such, if $r\geq 2$ the claim follows. Thus, we can assume that $r=1$. The Jorge--Meeks formula \eqref{eq:jm} yields
\[
g-1 + \frac{1}{2}  (d_1 + 1) = \deg(N) \geq 1
\]
Parity consideration implies that $d_1$ is odd. If $d_1 \geq 3$, the lemma follows, so we thus find $d_1=1$. In particular $X$ has precisely one end, and that end is embedded. An embedded end of finite total curvature is either asymptotic to a plane or a catenoid. The half-space theorem of Hoffman--Meeks \cite{HoffmanMeeks:halfspace} implies that $X(\Sigma)$ is flat (cf.\ \cite[Remark 2.2]{HoffmanKarcher}), a contradiction.

When $X$ is a one-sided immersion, we can assume that $r=1$ as above. In this case, \eqref{eq:jmnon} yields
\[
g-1 + d_1 + 1 = \deg(n) \geq 1.
\]
As such, we must consider $d_1\in\{1,2\}$. The case that $d_1=1$ violates the half-space theorem as before. For $d_1=2$, we can appeal to the monotonicity formula to conclude that $X$ is an embedding since $X(\Sigma)$ has density $2$ at $\infty$ (see \cite[Proposition 3.1]{HoffmanKarcher}). This contradicts the one-sidedness of $X$ since a proper embedding into $\RR^3$ is two-sided. 
\end{proof}
\subsection{Distance function at an end $E_j$}\label{sec:distend} We finish this part of the paper with Sections \ref{sec:distend} through \ref{sec:test}. They carry important calculations that will be used later, but could be skipped at first reading. We assume two-sidedness again until Section \ref{sec:nonor} (where we will apply these computations to the two-sided double cover). 

The following calculation will be important for proving integrability. Write $X: D_j\backslash\{0\}\rightarrow \RR^{3}$ for a conformal parametrization of one of its end $E_j$ with multiplicity $d_{j}$. Assume that the Gauss map satisfies $g(0) = 0$. If $z$ is the local coordinate in $D_j\backslash\{0\}$, we follow the proof of \cite[Lemma 10]{CM16} (cf.\ \cite[Proposition 2.1]{HoffmanKarcher}) and write the Weierstrass representation: 
\begin{align*}
X(z)=\re\displaystyle \int (\phi_1,\phi_2,\phi_3) 
\end{align*}
\noindent where 
\begin{align*}
\phi_{1}(z) & = \left( \frac{A}{z^{d_j+1}} + \frac{B_{1}(z)}{z^{d_j}} \right) dz\\
\phi_{2}(z) & = \left( \frac{iA}{z^{d_j+1}} + \frac{B_{2}(z)}{z^{d_j}} \right) dz\\
\phi_{3}(z) & = \left( \frac{B_{3}(z)}{z^{d_j}} \right) dz,
\end{align*}
for $B_{1}(z),B_{2}(z),B_{3}(z)$ holomorphic on $D_{j}$ and $A \in \CC\setminus\{0\}$. Integrating this, (one must slightly modify the estimate when $d_j=1$ but the conclusions below can be seen to still hold) we see that 
\[
X(z) = \re \left(-\frac{A}{d_{j}} z^{-d_{j}} + O(z^{-d_{j}+1}),-\frac{iA}{d_{j}} z^{-d_{j}} + O(z^{-d_{j}+1}), O(z^{-d_{j}+1}) \right),
\]
as $z\to 0$, so
\begin{equation}\label{eq:dist-end-spinning}
|X(z)| \simeq |z|^{-d_{j}}.
\end{equation}

Continuing on, because $g(z) = O(z)$, the unit normal can be written as
\begin{align*}
N=\left(\frac{2\re(g)}{|g|^2+1}, \frac{2\imag(g)}{|g|^2+1},\frac{|g|^2-1}{|g|^2+1}\right) = (0,0,-1) + O(z).
\end{align*}
This allows us to estimate $N\cdot \nabla_{\mathbb{R}^3} |x|$ as $z\rightarrow 0$,
\[
N\cdot \nabla_{\mathbb{R}^3} |x| = N \cdot \frac{X(z)}{|X(z)|} = O(z).
\]
Since $|X(z)|\simeq |z|^{-d_j}$, we find that
\begin{equation}\label{eqend}
|N\cdot \nabla_{\mathbb{R}^3} |x| | \leq C |z| \leq C |X(z)|^{-1/d_j}.
\end{equation}

\subsection{The Gaussian curvature at an end $E_{j}$} Assuming as in the previous section that $X:D_{j}\setminus\{0\}\to\RR^{3}$ is conformally parametrizing the end $E_{j}$. Write $d_{j}$ for the multiplicity of $E_{j}$ and assume that $g(0) = 0$. Write
\[
B_{3}(z)= b z^{n} + O(z^{n+1})
\]
as $z\to 0$, for some $n\geq 0$ and $b\not = 0$.

Recall that the Weierstrass representation gives (see \cite[Theorem 2.1]{HoffmanKarcher})
\[
(\phi_{1}(z),\phi_{2}(z),\phi_{3}(z)) = \left(\frac 12 (g^{-1}-g)dh, \frac i2(g^{-1}+g)dh,dh \right),
\]
for $\phi_{i}$ considered above. Note that 
\[
g=\dfrac{\phi_3}{\phi_1-i\phi_2}=\dfrac{B_3(z)z}{2A+z(B_1-iB_2)} = \frac{b}{2A} z^{n+1}  + O(z^{n+2})
\]
and
\[
dh = b z^{-d_{j}+n} + O(z^{-d_{j}+n+1})
\]
as $z\to 0$. 

Now, recall that the Gaussian curvature can be written in terms of $g$ and $dh$ (see \cite[(2.10)]{HoffmanKarcher}) as
\[
\kappa = \frac{-16}{(|g|+|g|^{-1})^{4}} \left| \frac{dg/g}{dh} \right|^{2}.
\]
We compute
\[
\left|\frac{dg/g}{dh}\right|^{2} = \left(\frac{n+1}{b}\right)^{2} |z|^{2(d_{j}-n-1)} (1+O(z))
\]
Thus, there is $c>0$ depending on $E_{j}$ so that
\[
\kappa(X(z)) = - c |z|^{2(d_{j}+n+1)}(1+O(z))
\]
as $z\to 0$. It is useful to write this using $|X(z)|\simeq |z|^{-d_{j}}$ as 
\begin{equation}\label{eq:gauss-curv-end}
\kappa(x) = O( |x|^{-2- \frac{2(n+1)}{d_{j}}}). 
\end{equation}
for $x \in E_{j}$ with $|x|\to\infty$. All that will matter in the sequel is the super-quadratic decay.

\subsection{A smoothed out log-cutoff function}\label{sec:test}
Let $\xi:\mathbb{R}\rightarrow\mathbb{R}$ be a smooth function satisfying $\xi(s) = 0$ for $s\leq 0$ and $\xi(s) = 1$ for $s\geq 1$. In particular, there exists a constant $C>0$ such that $|\xi'|,|\xi''|\leq C$. For any $R>1$, we define $\varphi_R:\Sigma\rightarrow \mathbb R$ to be:
\[
\varphi_R(x) = \xi \circ \psi_R(x),
\]
where $\psi_R(x)=2 -\frac{\log |x|}{\log R}$ is the standard log-cutoff function.

Note that $\varphi_R$ satisfies $\varphi_{R}(x) = 1$ for $|x|\leq R$ and  $\varphi_{R}(x) = 0$ for $|x|\geq R^{2}$. The goal of the remainder of this section is to estimate $|\nabla \varphi_R|$ and $|\Delta  \varphi_R|$ where $\nabla,\Delta$ are the intrinsic gradient and Laplacian associated to $X^{*}(g_{\RR^{3}})$. These are clearly supported in the compact region $X^{-1}(B_{R^{2}}(0)\setminus B_{R}(0))$. First, we note:
\begin{align}\label{est:gradient}
|\nabla \varphi_R| &= |\xi'|  |\nabla \psi_R| \nonumber \\
& \leq C \left| \frac{1}{|x|\log R} \nabla |x|\right| \nonumber \\
& \leq  \frac{C}{|x|\log R}.
\end{align}
Next, since $\Sigma$ is minimal, we recall that $|x|\Delta |x| = 2 - |\nabla |x||^2$. Hence: 
\begin{align*}
\Delta \log |x|& = \frac{1}{|x|} \Delta |x| -\frac{1}{|x|^2} |\nabla |x||^2 \\
&= \frac{2}{|x|^2} (1-|\nabla |x||^2)
\end{align*}
Since the gradient of $|x|$ in $\Sigma$ is just the tangent projection of the gradient of $|x|$ in $\mathbb{R}^3$, we have that $1-|\nabla |x||^2=(N \cdot \nabla_{\RR^{3}} |x|)^2$. Using this: 
\[
\Delta \psi_R(x) =
-2\frac{(N\cdot \nabla_{\RR^{3}} |x|)^2}{|x|^2\log R},
\]
and thus, by \eqref{eqend}, if $k=\min\{d_1,d_2,\ldots,d_r\}$, we get the estimate:
\begin{equation*}
|\Delta \varphi_R (x)| \leq \frac{C}{ |x|^{2+2/k}\log R}.
\end{equation*}
Therefore, we finally obtain (taking $C>0$ larger if necessary):
\begin{align}\label{est:laplacian}
|\Delta \varphi_R | \leq |\xi''||\nabla \varphi_R|^2 + |\xi'| |\Delta \varphi_R| 
\leq \frac{C}{|x|^2\log^2R} + \frac{C}{ |x|^{2+2/k}\log R}.
\end{align}

\section{A new weighted space and more harmonic 1-forms}\label{sec:weights}

We define the weighted space $L^{2}_{\ast}(\Sigma)$ to be the completion of smooth compactly supported functions with respect to the norm
\begin{equation*}
\Vert f \Vert_{L^{2}_{\ast}(\Sigma)}^{2} : = \int_{\Sigma} f^{2} (1+|x|^2)^{-1}(\log(2+|x|))^{-2},
\end{equation*}
where $|x|$ is the Euclidean distance. This norm clearly comes from an inner product, making $L^{2}_{\ast}(\Sigma)$ into a Hilbert space.

\subsection{Harmonic 1-forms}
We will show next that the extra logarithmic weight allows for more harmonic 1-forms. Write $\sH^{1}(\Sigma)$ for the space of all harmonic $1$-forms on $\Sigma$. 
\begin{prop}
If  $X: D\backslash\{0\}\rightarrow \Sigma$ be a conformal parametrization of an end $E_j$ of $\Sigma$. If $z$ is the local coordinate in $D\backslash\{0\}$, the 1-forms $ \frac {dz} z, \frac {dz} {z^2}, \ldots, \frac{dz}{z^{d_j+1}}$ are all integrable in $L^2_\ast(E_{j})$, where $d_{j}$ is the multiplicity of $E_{j}$. 
\end{prop}

\begin{proof}
Let $\omega=\frac{dz}{z^l}$, where $1\leq l\leq d_j+1$. By \eqref{eq:dist-end-spinning}, we have that $|X(z)| \simeq |z|^{-d_{j}}$. Additionally, because the squared norm of a 1-form times the area element is a pointwise conformally invariant quantity, we can estimate:
\begin{align*}
\left\Vert\frac{dz}{z^l}\right\Vert^2_{L^2_\ast(E_j)} &\leq \int_{D_j\backslash\{0\}} \frac{1}{|z|^{2l}} (1+|X(z)|^2)^{-1}(\log(2+|X(z)|))^{-2} |dz|^2 \\
& \leq C \int_{D_j\backslash\{0\}} \frac{1}{|z|^{2l}}\frac{|z|^{2d_j}}{1+|z|^{2d_j}}\frac{1}{\log^2(2+|z|^{-d_j})} |dz|^2\\
& \leq C \int_{0}^{1} r^{2(d_{j}-l) +1} (\log r)^{-2} \, dr
\end{align*}
For any $l\in\{1,\dots,d_j+1\}$, the above integral is finite.\footnote{Note that $\left( \frac{1}{\log r} \right)' = - r^{-1}(\log r)^{-2}$. (This is precisely the standard log-cutoff trick.)}
\end{proof}

\begin{rema}\label{rema:opt}  In the case of a conformal parametrization $X: D\backslash\{0\}\rightarrow \Sigma$ of an embedded end $E$ of $\Sigma$, the form $\frac{dz}{z^2}$ is actually equal to $-dw$, where $w=\frac{1}{z}$ has bounded geometric norm $|dw|$ ($|\cdot|$ being the norm taken with respect to the Riemannian metric on $\Sigma$). Thus, $\left| \frac{dz}{z^2}\right|$, is a bounded function on $\Sigma$ and thus can be seen to belong in $L^2_\ast$ after adding the logarithmic weight. On the other hand, $\left| \frac{dz}{z^3}\right| = \left| wdw\right|$ is unbounded on $\Sigma$.

\end{rema}

\begin{prop}\label{prop:dim}
If $\Sigma$ has genus $g$ and $r$ ends, with multiplicities $d_1,d_2,\ldots, d_r$, the dimension of $\mathscr{H}^1(\Sigma) \cap L^2_\ast(\Sigma)$, i.e., harmonic forms in $L_{*}^{2}(\Sigma)$, is $2g+ 2\sum_{j=1}^{r}(d_j+1) -2$.
\end{prop}

\begin{proof}
	By Osserman \cite{Osserman:FTC}, $\Sigma$ is conformal to a Riemann surface with genus $g$ and $r$ punctures, say $p_1,p_2\ldots,p_r$,  corresponding respectively to the ends $E_1,E_2,\ldots E_r$. By Riemann-Roch,\footnote{Specifically, this follows by considering the divisor $\mathfrak{U} = \prod_{j=1}^{r} p_{j}^{-d_{j}-1}$ in e.g., \cite[\S III.4]{FK}. } the dimension of holomorphic 1-forms on $\overline{\Sigma}$ which have poles of order at most $d_j+1$ at the point $p_j$, for $j=1,\ldots,r$ is exactly
\[
g+ \sum_{j=1}^{r}(d_j+1) -1.
\]
By taking real and imaginary parts, we find at least $2g+ 2\sum_{j=1}^{r}(d_j+1) -2$ harmonic forms in $L^{2}_{*}(\Sigma)$. 

Conversely, for any harmonic form $\omega\in L^{2}_{*}(\Sigma)$, the form $\widetilde{\omega} : = \omega + i * \omega$ is holomorphic on $\Sigma$ (cf.\ \cite[I.3.11]{FK}). Using \eqref{eq:dist-end-spinning} and the definition of $L^{2}_{*}(\Sigma)$, it is easy to check that $\widetilde{\omega}$ must extend meromorphically to $\overline\Sigma$, with a pole of order $d_{j}+1$ at each point $p_{j}$. This proves the assertion. 
\end{proof}

\section{Weighted eigenfunctions}\label{sec:weightedFC}

\begin{prop}\label{prop:gradientl2}
	Suppose $f\in L^2_\ast(\Sigma)$ is a smooth function on $\Sigma$ satisfying the weighted eigenvalue equation: 
	$$\Delta f -2\kappa f + \lambda (1+|x|^2)^{-1}(\log(2+|x|))^{-2} f=0.$$
	Then $|\nabla f| \in L^2(\Sigma)$.
\end{prop}

\begin{proof}
Let $\varphi= \varphi_R$ as defined in Section \ref{sec:test}. We will prove that $\int_{\Sigma} \varphi^2 |\nabla f|^2 $ is uniformly bounded as $R\rightarrow \infty$.
First we compute:
\begin{align*}
\int_{\Sigma} \varphi^2 |\nabla f|^2 &= - \int_\Sigma \varphi^2 f \Delta f - \frac{1}{2} \int_{\Sigma} \nabla \varphi^2 \cdot \nabla f^2 \\
& = - \int_{\Sigma} 2\kappa \varphi^2 f^2 + \lambda \int_{\Sigma} \varphi^2 f^{2} (1+|x|^2)^{-1}(\log(2+|x|))^{-2} -\frac{1}{2} \int_{\Sigma} \nabla \varphi^2 \cdot \nabla f^2
\end{align*}
Since $\kappa$ decays super-quadratically by \eqref{eq:gauss-curv-end}, and $f\in L^2_\ast$ we only need to check the integrability of $\int_{\Sigma} \nabla \phi^2 \cdot \nabla f^2$ as $R\rightarrow\infty$. Then:
\begin{align}\label{eq:aux}
-\int_{\Sigma} \nabla \varphi^2 \cdot \nabla f^2 &=\int_{\Sigma} f^2 \Delta \varphi^2\nonumber\\
& = 2\int_{\Sigma} f^2 \varphi\Delta \varphi  + 2 \int_{\Sigma} f^2 |\nabla \varphi|^2
\end{align}
We analyze each term separately. First note that, by \eqref{est:laplacian}:
\begin{align}\label{eq:laplacian}
\left|\int_{\Sigma} f^2 \varphi\Delta \varphi \right| &\leq \int_{\Sigma} f^2 |\varphi||\Delta \varphi| \nonumber\\
&\leq C \underbrace{ \int_{\Sigma \cap (B_{R^2}\setminus B_R)} f^2 \frac{1}{|x|^2\log^2R}}_{(\textrm{I})}  + C \underbrace{\int_{\Sigma  \cap (B_{R^2}\setminus B_R) } f^2\frac{1}{  |x|^{2+2/k}\log R}}_{(\textrm{II})}.
\end{align}
where $k\geq1$ is the highest multiplicity of an end of $\Sigma$.
        	
Hence:
\begin{align*}
(\textrm{I})=\int_{\Sigma  \cap (B_{R^2}\setminus B_R) } f^2 \frac{1}{|x|^2\log^2R}  
& = \int_{\Sigma  \cap (B_{R^2}\setminus B_R) } f^2 \frac{\log^2|x|}{|x|^2 \log^2|x|  \log^2 R}\\
& \leq \int_{\Sigma  \cap (B_{R^2}\setminus B_R) } f^2 \frac{\log^2R^2}{|x|^2 \log^2|x|  \log^2 R}\\
& \leq 4 \int_{\Sigma\cap (B_{R^2}\setminus B_R)} f^2 \frac{1}{|x|^2 \log^2|x| }.
\end{align*}
This tends to zero as $R\to\infty$, since $f \in L^{2}_{\ast}$. 
    	
Furthermore, we find:
\begin{align*}
(\textrm{II})=\int_{\Sigma  \cap (B_{R^2}\setminus B_R) } f^2\frac{1}{|x|^{2+2/k}\log R} \leq 2 \int_{\Sigma\cap (B_{R^2}\setminus B_R)} f^2 \frac{1}{|x|^{2} \log^2|x|}  \frac{\log |x|}{  |x|^{2/k}},
\end{align*}
which also goes to zero as $R\rightarrow \infty$, since $f\in L^2_\ast$ and $|x|^{2/k}\gg\log |x|$ as $x\rightarrow\infty$.

Finally, for the second term in \eqref{eq:aux}:
\begin{align*}
\int_{\Sigma} f^2 |\nabla \varphi|^2  & \leq C \int_{\Sigma\cap(B_{R^2}\setminus B_R)} f^2 \frac{1}{|x|^2\log^2 R}\\
& \leq C \int_{\Sigma\cap(B_{R^2}\setminus B_R)} f^2 \frac{\log^2|x|}{|x|^2 \log^2|x|  \log^2 R}\\
& \leq C \int_{\Sigma\cap(B_{R^2}\setminus B_R)} f^2 \frac{\log^2R^2}{|x|^2 \log^2|x|  \log^2 R}\\
& \leq C \int_{\Sigma\cap(B_{R^2}\setminus B_R)} f^2 \frac{1}{|x|^2 \log^2|x| }.
\end{align*}
This tends to zero as $R\rightarrow\infty$, like before. This completes the proof. 
\end{proof}

\begin{lemm}\label{lem:forml2}
For $\omega$ a 1-form in $L^2_\ast(\Sigma) \cap \sH^1(\Sigma)$, we have that 
$$\int_{\Sigma}|\nabla \omega|^2 <\infty.$$
\end{lemm}
\begin{proof}
Using the Bochner formula as in Lemma 13 of \cite{CM16}, the proof follows similarly to Proposition \ref{prop:gradientl2} above. 	
\end{proof}

We now prove that the weighted analogue of Fischer--Colbrie's result constructing the index on a complete minimal surface \cite[Proposition 2]{Fischer-Colbrie:1985} holds for $L^{2}_{\ast}(\Sigma)$. 
\begin{prop}\label{prop:weighted-FC}
	Suppose that $\Sigma$ has index $\Index(\Sigma)=k < \infty$ in the usual $L^{2}$-sense. Then there exists a $k$-dimensional subspace $W$ of $L^{2}_{\ast}(\Sigma)$ with an $L^{2}_{\ast}$-orthornormal basis of $L^{2}_{\ast}$-eigenfunctions for the stability operator $f_{1},\dots,f_{k}$. Letting the associated eigenvalues be $\lambda_{1},\dots,\lambda_{k}$, it holds that $\lambda_{i} < 0$. 
	
Moreover, we have $Q(\phi,\phi) \geq 0$ for $\phi \in C^{\infty}_{0}(\Sigma) \cap W^{\perp}$, where $W^{\perp}\subset L^{2}_{\ast}(\Sigma)$ is the $L^{2}_{\ast}(\Sigma)$ orthogonal complement of $W$. 
\end{prop}
\begin{proof}
%We will adapt the arguments in \cite[Proposition 8]{CM16}. 
Choose $R_0$ sufficiently large so that $k=\Index(\Sigma) = \Index (\Sigma\cap B_{\rho})$ for $\rho > R_{0}$. Let $\{f_{1,\rho},\dots,f_{k,\rho}\}$ and $\{\lambda_{1,\rho},\dots,\lambda_{k,\rho}\}$ denote the $L^{2}_{\ast}(\Sigma \cap B_{\rho})$-weighted Dirichlet eigenfunctions and eigenvalues, respectively, of an $L^{2}_{\ast}(\Sigma \cap B_{\rho})$-orthonormal basis constructed by minimizing the Rayleigh quotient 
\begin{equation*}
Q(\phi,\phi)/\Vert\phi\Vert^{2}_{L^{2}_{\ast}(\Sigma\cap B_{\rho})}
\end{equation*}
on $C^\infty_c(\Sigma\cap B_\rho)$ (note that there are exactly $k$ such eigenfunctions since the $L^2$ and $L^2_{\ast}$ norms are equivalent in $B_{\rho}$). It is not hard to check that this implies that
\begin{equation*}
\Delta f_{i,\rho} - 2\kappa f_{i,\rho} + \lambda_{i,\rho}(1+|x|^{2})^{-1}(\log(2+|x|))^{-2} f_{i,\rho} = 0
\end{equation*}
in $\Sigma \cap B_\rho$ and $f=0$ on $\Sigma \cap \partial B_\rho$.

Our goal is to prove that $f_{i,\rho}$ converges to the desired eigenfunctions as for an appropriate sequence of $\rho$ going to infinity. To that end, note that $\Sigma\setminus B_{R}$ must be stable for any $R>R_0$. Thus, if $\eta$ is any smooth function vanishing in $B_R$, the stability inequality implies that for any $\phi\in C^{\infty}_{c}(\Sigma)$,
\begin{equation}\label{eq:outer-stable}
-\int_{\Sigma} 2\kappa (\eta\phi)^{2}\leq \int_{\Sigma}|\nabla (\eta\phi)|^{2} = \int_{\Sigma} \eta^{2}|\nabla\phi|^{2}+ 2\eta\phi\bangle{\nabla\eta,\nabla\phi}+ \phi^{2}|\nabla\eta|^{2}.
\end{equation}	
In particular, arguing this as in \cite[Proposition 2]{Fischer-Colbrie:1985} and \cite[Proposition 8]{CM16}, choosing $\eta=1-\varphi_R$,
\begin{align*}
\int_{\Sigma} (1-\eta^{2})(|\nabla \phi|^{2} + 2\kappa \phi^{2}) & \leq Q(\phi,\phi) + \int_{\Sigma} 2 \eta\phi \bangle{\nabla \eta,\nabla \phi} + \phi^{2} |\nabla \eta|^{2}\\
& \leq Q(\phi,\phi) + \int_{\Sigma\cap B_{R^{2}}} \eta^{2} |\nabla \phi|^{2} + 2\phi^{2} |\nabla \eta|^{2}\\
& \leq 2 Q(\phi,\phi) + \int_{\Sigma\cap B_{R^{2}}} 2\phi^{2} |\nabla \eta|^{2} - 2\kappa \phi^{2}.
\end{align*}
Thus, we find
\begin{equation}\label{eq:W12-bd-Q}
\int_{\Sigma \cap B_{R}} |\nabla\phi|^{2}  \leq 2Q(\phi,\phi) +\underbrace{\left( \frac{C}{R^2\log^2R}+ \sup_{\Sigma \cap B_{R^2}}4|\kappa|\right)}_{:=2C_{R}} \int_{\Sigma\cap B_{R^2}}  \phi^{2}.
\end{equation}
Because $\max\{\lambda_{1,\rho},\dots,\lambda_{k,\rho}\}$ is decreasing with $\rho$, there is $\epsilon_{0}>0$ so that for $\rho\geq R>R_{0}$, we have $\lambda_{i,\rho} < -\epsilon_{0}$. On the other hand, the inequality we have just proven shows that $\lambda_{i,\rho}\geq -C_{R} (1+R^2)\log^2(2+R^{2})$. 

Moreover, \eqref{eq:W12-bd-Q} yields
\begin{equation*}
\int_{\Sigma \cap B_{R}} f_{i,\rho}^{2} + |\nabla f_{i,\rho}|^{2} \leq (1+2C_{R}) \int_{\Sigma\cap B_{R^2}} f_{i,\rho}^{2} \leq (1+2C_{R}) (1+R^2)\log^2(2+R^{2}).
\end{equation*}
From this, as in \cite{Fischer-Colbrie:1985,CM16}, for every $i$ fixed, a diagonal argument (together with the fact that $W^{1,2}(\Sigma\cap B_{R})$ compactly embeds into $L^{2}_{\ast}(\Sigma\cap B_{R})$) implies the existence of a sequence $\rho_j\rightarrow\infty$ and a function $f_{i}\in L^2_\ast(\Sigma)$ such that 
\[
f_{i,\rho_j} \rightarrow f_{i}\quad\text{in}\,\, L^2_\ast(K)
\]
for any compact domain $K$ of $\Sigma$. Because $\lambda_{i,\rho}$ is uniformly bounded from below, it is clear that $f_{i}$ is an $L^{2}_{*}$-eigenfunction with eigenvalue $\lambda_{i} = \lim_{j\to\infty}\lambda_{i,\rho_{j}}<0$.

To finish our proof we must show that no mass escapes to infinity, i.e., that $\Vert f_{i} \Vert_{L^2_\ast(\Sigma)}=1$. To this end, we return to \eqref{eq:outer-stable} and plug in $\phi = f_{i,\rho}$ (extending $\phi$ to be zero outside of $B_{\rho}$) and consider a smooth function $\eta$ vanishing in $B_R$ (for some $R>R_0$) to be chosen below. We compute:
\begin{align*}
-\int_{\Sigma} 2\kappa (\eta f_{i,\rho})^{2} &\leq  \int_{\Sigma} \eta^{2}|\nabla f_{i,\rho}|^{2}+ 2\eta f_{i,\rho}\bangle{\nabla\eta,\nabla f_{i,\rho}}+ f_{i,\rho}^{2}|\nabla\eta|^{2}\\
&=  \int_{\Sigma} \eta^{2}|\nabla f_{i,\rho}|^{2}+ \frac 1 2\bangle{\nabla\eta^{2},\nabla f_{i,\rho}^{2}}+ f_{i,\rho}^{2}|\nabla\eta|^{2}\\
&=  \int_{\Sigma} \eta^{2}|\nabla f_{i,\rho}|^{2}- \eta^{2}\left(  f_{i,\rho}\Delta  f_{i,\rho} + |\nabla  f_{i,\rho}|^{2} \right)+ f_{i,\rho}^{2}|\nabla\eta|^{2}\\
&=  \int_{\Sigma}  f_{i,\rho}^{2}|\nabla\eta|^{2}- \eta^{2} f_{i,\rho}\Delta  f_{i,\rho}\\
&=  \int_{\Sigma}  f_{i,\rho}^{2}|\nabla\eta|^{2}- 2\kappa(\eta f_{i,\rho})^{2} + \lambda_{i,\rho} (1+|x|^{2})^{-1}(\log(2+\log|x|))^{-2}(\eta f_{i,\rho})^{2}.
\end{align*}
Hence,
\begin{equation}\label{eq:cutoff-caccioppoli}
\epsilon_{0}\Vert \eta f_{i,\rho}\Vert_{L^{2}_{\ast}(\Sigma)}^{2} \leq (-\lambda_{i,\rho}) \Vert \eta f_{i,\rho}\Vert_{L^{2}_{\ast}(\Sigma)}^{2} \leq \int_{\Sigma}  f_{i,\rho}^{2}|\nabla\eta|^{2}
\end{equation}
We wish to choose
\[
\eta(x) = \begin{cases} 0 & |x| \leq R \\  \frac{1}{\log \beta} \log \left( \frac{\log |x|}{\log R} \right) &  R < |x| < e^{\beta} R \\ 1 & |x| \geq e^\beta R, \end{cases}
\]
where $\beta > 1$ will be fixed large below. To justify this choice of $\eta$, note that we can choose $\eta_\ell \in C^\infty(\Sigma)$ with $\supp \eta_\ell \subset \Sigma \setminus B_R$ so that $\eta_\ell \to \eta$ and $\nabla \eta_\ell \to \nabla \eta$ almost everywhere. Because $\eta_\ell$ is a valid choice in \eqref{eq:cutoff-caccioppoli}, we can send $\ell\to\infty$ using dominated convergence. This justifies taking the given $\eta$ in \eqref{eq:cutoff-caccioppoli}. 

Now, we note that for $R < |x| < e^\beta R$, 
\[
|\nabla \eta| \leq \frac{1}{\log \beta} \frac{1}{ |x| |\log |x|} 
\]
Combined with \eqref{eq:cutoff-caccioppoli}, we thus find 
\[
\Vert f_{i,\rho} \Vert_{L^2_\ast(\Sigma\setminus B_{e^\beta R})} \leq \frac{C}{(\log\beta)^2} \Vert f_{i,\rho}\Vert_{L^2_\ast(\Sigma)} = \frac{C}{(\log\beta)^2} 
\]
Now, for $\delta>0$ fixed, take $\beta$ sufficiently large so that $\frac{C}{(\log\beta)^2} < \delta$. Then, we have a compact set $K=\overline{\Sigma \cap B_{e^\beta R}}$ so that
\[
||f_{i,\rho_j}||_{L^2_\ast(\Sigma\setminus K)} < \delta
\]
for all $i \in \{1,\dots,k\}$ and $j$ sufficiently large so that $\rho_j\geq e^{\beta} R$. This proves that $f_{i,\rho_{j}}$ does not loose any mass at infinity. The proof may now be completed as in \cite[p.\ 126]{Fischer-Colbrie:1985}. 
\end{proof}

\section{Proof of Theorem \ref{theo:main}}\label{sec:proof-main}

We first recall the notation of  \cite{CM16}. For $\omega$ harmonic 1-form, we write:
\begin{equation*}
X_{\omega} : = (\bangle{\omega,dx_{1}},\bangle{\omega,dx_{2}},\bangle{\omega,dx_{3}}).
\end{equation*}
By \cite[Lemma 1]{Ros:one-sided}, we have
\begin{equation*}
\Delta X_{\omega} -2\kappa X_{\omega} = 2 \bangle{ \nabla \omega, h} N,
\end{equation*}
where this equation is to be interpreted in the component by component sense. Here, $h$ is the second fundamental form. Moreover, if $\bangle{\nabla \omega, h} \equiv 0$, then $\omega\in \Span\{*dx^{1},*dx^{2},*dx^{3}\}$ by \cite[Lemma 1]{Ros:one-sided}.

As in \cite{CM16}, for vector fields $X$ and $Y$ along $\Sigma$ with components $(X_{1},X_{2},X_{3})$ and $(Y_{1},Y_{2},Y_{3})$, we denote by $Q(X,Y)$ the sum $\sum_{i=1}^{3}Q(X_{i},Y_{i})$. 
%Moreover, for a function $f$, we will denote by $Q(X,f) : = \sum_{i=1}^{3}Q(X_{i},f)$. 
For example, for a vector field $X$ along $\Sigma$, we have that 
\begin{equation*}
Q(X,X) = - \int_{\Sigma} \bangle{\Delta X - 2\kappa X,X},
\end{equation*}
where the integrand is the Euclidean inner product of the following vector fields along $X$
\begin{equation*}
(\Delta X_{1}-2\kappa X_{1},\Delta X_{2}-2\kappa X_{2},\Delta X_{3}-2\kappa X_{3}) \qquad \text{and} \qquad (X_{1},X_{2},X_{3}).
\end{equation*}

By Proposition \ref{prop:weighted-FC}, there are $L^2_\ast$-eigenfunctions $f_{1},\dots,f_{k} \in L^{2}_{\ast}(\Sigma)$ which span $W \subset L^{2}_{\ast}(\Sigma)$ and so that for $\phi \in C^{\infty}_{0}(\Sigma)\cap W^{\perp}$, we have $Q(\phi,\phi)\geq0$. By Proposition \ref{prop:dim}, $V := L^{2}_{\ast}(\Sigma)\cap\sH^{1}(\Sigma)$ has dimension 
\[
2g+2\displaystyle\sum_{j=1}^{r}(d_j+1)-2
\] 
Suppose that $\omega \in V$ satisfies $ X_{\omega} \in W^{\perp}$ (where the orthogonal complement is taken with respect to the $L^{2}_{\ast}$-inner product). We will prove that $\omega \in \Span \{\mbox{$*$}dx_{1},\mbox{$*$}dx_{2},\mbox{$*$}dx_{3}\}$. This finishes the proof (cf.\ \cite[\S 5]{CM16}).

Suppose that $\omega \in V$ satisfies $X_{\omega}\in W^{\perp}$. Following \cite{Ros:one-sided,CM16}, we will show that $Q(X_{\omega},Y) = 0$ for any compactly supported smooth vector field $Y$ with $Y \in W^{\perp}$. Choose $R$ sufficiently large so that $B_{R}$ contains the support of $Y$. We set
\begin{equation*}
X_{t} : = \varphi_{R}(X_{\omega} + tY + f_{1}\vec{c}_{1}+\dots+f_{k}\vec{c}_{k}),
\end{equation*}
where $\varphi_{R}$ is the test function constructed in Section \ref{sec:test}. Here, the vectors $\vec{c}_{j} \in \RR^{3}$ depend on $X_{\omega}$, $\varphi_{R}$ and are chosen so that $X_{t} \in W^{\perp}$. In particular, this condition means that 
\begin{equation*}
\int_{\Sigma} \varphi_{R}(X_{\omega} +  f_{1} \vec{c}_{1}+\dots+f_{k}\vec{c}_{k})f_{j} (1+|x|^2)^{-1}(\log(2+|x|))^{-2} = 0,
\end{equation*}
where we have used the fact that $Y \in W^{\perp}$ and $\varphi_{R} Y = Y$. Because $X_{\omega},f_{j} \in L^{2}_\ast(\Sigma)$ and the $f_{1},\dots,f_{k}$ form an $L^{2}_{\ast}$-orthonormal basis for $W$, the dominated convergence theorem guarantees that the $\vec{c}_{j}$ tend to $0$ when sending $R\rightarrow\infty$.

As in \cite{Ros:one-sided,CM16}, the stability inequality implies that:
\begin{align}
\notag Q(X_{\omega},Y)^{2} & \leq Q(Y,Y)\\
\label{eq:I-II-III} & \times  \left(\underbrace{Q(\varphi_{R} X_{\omega},\varphi_{R} X_{\omega})}_{(\textrm{I})} + 2\sum_{i=1}^{k} \underbrace{Q(\varphi_{R} X_{\omega},\varphi_{R} f_{i}\vec{c}_{i})}_{(\textrm{II})}  + \sum_{i,j=1}^{k} \underbrace{Q(\varphi_{R} f_{i}\vec{c}_{i},\varphi_{R}  f_{j}\vec{c}_{j})}_{(\textrm{III})} \right)
\end{align}
We claim that the term in parenthesis tends to zero as  $R\rightarrow\infty$. To show this, we consider each term in \eqref{eq:I-II-III} separately. We have 
\begin{align*}
(\textrm{I}) = Q(\varphi_{R} X_{\omega},\varphi_{R} X_{\omega}) & = \int_{\Sigma} (|\nabla (\varphi_{R} X_{\omega})|^{2} + 2\kappa (\varphi_{R})^{2} |X_{\omega}|^{2})\\
& = - \int_{\Sigma} \bangle{ \Delta (\varphi_{R} X_{\omega}) - 2\kappa \varphi_{R} X_{\omega}, \varphi_{R} X_{\omega}}\\
& = - \int_{\Sigma} (\varphi_{R})^{2} \bangle{ \Delta  X_{\omega} - 2\kappa X_{\omega}, X_{\omega}} \\
& \qquad  - \int_{\Sigma}  ( \varphi_{R}\Delta \varphi_{R} |X_{\omega}|^{2} + 2 \bangle{d \varphi_{R} \otimes  X_{\omega}, \varphi_{R} \nabla X_{\omega}}) \\
& = - \int_{\Sigma} \left( \varphi_{R}\Delta \varphi_{R} |X_{\omega}|^{2} + \frac 12  \bangle{\nabla (\varphi_{R})^{2}, \nabla |X_{\omega}|^{2}} \right) \\
& = \int_{\Sigma} |\nabla \varphi_{R}|^{2} |X_{\omega}|^{2}.
\end{align*}
Since $|\nabla \varphi_R|\leq \frac{C}{|x|\log R}$, we have:
\begin{align*}
\int_{\Sigma} |\nabla \varphi_{R}|^{2} |X_{\omega}|^{2}& \leq \int_{\Sigma\cap (B_{R^2}\setminus B_R)}|X_{\omega}|^{2} \frac{C}{|x|^2\log^2 R}\\ 
& \leq {C}\int_{\Sigma\cap (B_{R^2}\setminus B_R)}|X_{\omega}|^{2} \frac{1}{|x|^2 \log^2|x|}\frac{\log^2|x|}{\log^2 R}\\
\intertext{Note that $\max_{r\in[R,R^{2}]}\frac{\log^2 r}{\log^2 R} = 4$ so taking $C>0$ larger, we find:}
& \leq {C}\int_{\Sigma\cap (B_{R^2}\setminus B_R)}|X_{\omega}|^{2} \frac{1}{|x|^2 \log^2|x|}\frac{\log^2|x|}{\log^2 R}\\
&\leq C\int_{\Sigma\cap (B_{R^2}\setminus B_R)}|X_{\omega}|^{2} \frac{1}{|x|^2 \log^2|x|}\\
&\leq C\int_{\Sigma\setminus B_R}|X_{\omega}|^{2} \frac{1}{|x|^2 \log^2|x|}
\end{align*}
Since $|X_\omega|\in L^2_\ast(\Sigma)$, this converges to $0$ as  $R\rightarrow \infty$.

For the second term, we compute:
\begin{align}
\notag (\textrm{II}) & = Q(\varphi_{R} X_{\omega},\varphi_{R} f_{i}\vec{c}_{i}) \\
\notag & = -\int_{\Sigma} \varphi_{R} \bangle{X_{\omega},\vec{c}_{i}}\left( \Delta(\varphi_{R} f_{i}) - 2\kappa \varphi_{R} f_{i} \right)\\
\notag & =  -\int_{\Sigma} (\varphi_{R})^{2} \bangle{X_{\omega},\vec{c}_{i}}\left( \Delta f_{i} - 2\kappa f_{i} \right) - \int_{\Sigma}\varphi_{R} \bangle{X_{\omega},\vec{c}_{i}}(\Delta \varphi_{R} f_{i} + 2\bangle{\nabla \varphi_{R},\nabla f_{i}})\\
\label{eq:I-123} & =  \lambda_{i} \int_{\Sigma} (\varphi_{R})^{2} \bangle{X_{\omega},\vec{c}_{i}} f_{i}(1+|x|^2)^{-1}(\log(2+|x|))^{-2}- \int_{\Sigma}\varphi_{R} \Delta \varphi_{R} \bangle{X_{\omega},\vec{c}_{i}} f_{i}  \\
& \notag \qquad - 2\int_{\Sigma} \varphi_{R}\bangle{X_{\omega},\vec{c}_{i}} \bangle{\nabla \varphi_{R},\nabla f_{i}}.
\end{align}
The first term in \eqref{eq:I-123} tends to zero as $R\to\infty$ by the dominated convergence and choice of $X_{\omega}\in W^{\perp}$. The second term in \eqref{eq:I-123} tends to zero as follows. By \eqref{est:laplacian}:
\begin{align*}
\left| \int_{\Sigma}\varphi_{R}\Delta \varphi_{R} \bangle{X_{\omega},\vec{c}_{i}} f_{i} \right| \leq&  |\vec{c}_{i}| \int_{\Sigma  \cap (B_{R^2}\setminus B_R) } \frac{C}{|x|^{2}\log^2R}|X_{\omega}||f_{i}| \\
&+ |\vec{c}_{i}| \int_{\Sigma  \cap (B_{R^2}\setminus B_R)} \frac{C}{  |x|^{2+2/k}\log R}|X_{\omega}||f_{i}|.
\end{align*}
Since $X_{\omega},f_{i} \in L^2_\ast(\Sigma)$, by arguing as in \eqref{eq:laplacian}, the above integrals tend to zero as $R\rightarrow\infty$.

For third term in \eqref{eq:I-123}, we use H\"older's inequality:
\begin{align*}
\left| \int_{\Sigma}\varphi_{R} \bangle{X_{\omega},\vec{c}_{i}} \bangle{\nabla \varphi_{R},\nabla f_{i}} \right| & \leq |\vec{c}_{i}| \left( \int_{\Sigma} |\nabla\varphi_{R}|^{2} |X_{\omega}|^{2}\right)^{\frac 12} \left( \int_{\Sigma} |\nabla f_{i}|^{2}\right)^{\frac 12}.
\end{align*}
Since $|\nabla f_i| \in L^2$ by Proposition \ref{prop:gradientl2}, $\vec{c}_{i} \rightarrow 0$ as $R\rightarrow\infty$, and $\int_{\Sigma} |\nabla\varphi_{R}|^{2} |X_{\omega}|^{2}$ also converges to zero, this also converges to zero as $R\to\infty$.  

Finally, we have that the third term in \eqref{eq:I-II-III} satisfies
\begin{align*}
(\textrm{III}) & = Q(\varphi_{R} \vec{c}_{i} f_{i},\varphi_{R} \vec{c}_{j} f_{j}) \\
& = -\frac 12 \bangle{\vec{c}_{i},\vec{c}_{j}} \int_{\Sigma} \varphi_{R} f_{j}(\Delta(\varphi_{R} f_{i}) -2\kappa \varphi_{R} f_{i})  - \frac 1 2 \bangle{\vec{c}_{i},\vec{c}_{j}}  \int_{\Sigma} \varphi_{R} f_{i}(\Delta(\varphi_{R} f_{j}) -2\kappa \varphi_{R} f_{j}) \\
& = \frac 12 (\lambda_{i} + \lambda_{j})\bangle{\vec{c}_{i},\vec{c}_{j}}  \int_{\Sigma}\varphi_{R}^{2} f_{i}f_{j}(1+|x|^2)^{-1}(\log(2+|x|))^{-2}  - \bangle{\vec{c}_{i},\vec{c}_{j}}  \int_{\Sigma}\varphi_{R}\Delta\varphi_{R} f_{i}f_{j} \\
& \qquad -\bangle{\vec{c}_{i},\vec{c}_{j}}  \int_{\Sigma} \varphi_{R} \bangle{\nabla \varphi_{R},  f_{i}\nabla f_{j} + f_{j}\nabla f_{i}}\\
& = \frac 12 (\lambda_{i} + \lambda_{j}) \bangle{\vec{c}_{i},\vec{c}_{j}} \int_{\Sigma}\varphi_{R}^{2} f_{i}f_{j}(1+|x|^2)^{-1}(\log(2+|x|))^{-2} +\bangle{\vec{c}_{i},\vec{c}_{j}}  \int_{\Sigma} |\nabla\varphi_{R}|^{2}f_{i}f_{j}.
\end{align*}
This tends to zero as $R\to\infty$ because the $\vec{c}_{i}$ are tending to zero and, by arguing as above, so does $\int_{\Sigma} |\nabla\varphi_{R}|^{2}|f_{i}f_{j}|$, since $f_{i},f_{j} \in L^2_\ast(\Sigma)$.

The above computations show that $Q(X_{\omega},Y) = 0$ for all compactly supported smooth vector fields $Y \in W^{\perp}$. Finally, fix an arbitrary smooth vector field $\tilde Y \in W^{\perp}$ and let 
\begin{equation*}
\tilde Y_{R} : = \varphi_{R}(\tilde Y +f_{1}\vec{c}_{1}+\dots+f_{k}\vec{c}_{k}),
\end{equation*}
where the $\vec{c}_{j}$ are chosen so that $\tilde Y_{R} \in W^{\perp}$. Observe (as above) that $\vec{c}_{j}\to 0$ as $R\to\infty$ by the dominated convergence theorem. Hence, because $\tilde Y_{R}$ has compact support, we have that
\begin{align*}
0 & = Q(X_{\omega},\tilde Y_{R}) \\
& = -\int_{\Sigma}\bangle{\Delta X_{\omega} - 2\kappa X_{\omega}, \tilde Y_{R}}\\
& = - 2 \int_{\Sigma}\bangle{\nabla \omega,h} \bangle{N,\tilde Y_{R}}.
\end{align*}
Lemma \ref{lem:forml2} shows that $|\nabla \omega|\in L^{2}(\Sigma)$, and because the second fundamental form satisfies $|h| \leq (1+|x|^{2})^{-\frac 12(1+\frac 1 k)}$ (cf.\ \eqref{eq:gauss-curv-end}), we may use the dominated convergence theorem to see that
\begin{equation*}
\int_{\Sigma}\bangle{\nabla \omega,h} \bangle{N,\tilde Y} = 0.
\end{equation*}
Similarly, we may show that for any vector $\vec{\alpha} \in \RR^{3}$ and eigenfunction $f_{i}\in W$ from Proposition \ref{prop:weighted-FC}, then
\begin{equation*}
\int_{\Sigma}\bangle{\nabla \omega,h} \bangle{N, f_{i}\vec{\alpha}} = 0.
\end{equation*}
Putting this together, we obtain $\bangle{\nabla\omega,h} = 0$. Thus, $\omega \in\Span \{\mbox{$*$}dx_{1},\mbox{$*$}dx_{2},\mbox{$*$}dx_{3}\}$. This completes the proof. 

\section{Proof of Theorem \ref{theo:nonor}}\label{sec:nonor}

We will prove Theorem \ref{theo:nonor} by essentially redoing the proof of Theorem \ref{theo:main}  for anti-invariant harmonic 1-forms in $L^2_\ast({\widehat{\Sigma}})$, where $\widehat{\Sigma}$ is the two-sided double cover of $\Sigma$. 

Note that the space $\mathscr{H}^1(\widehat{\Sigma})$ decomposes as 
$$\sH^1(\widehat{\Sigma}) = \sH_{+}^1(\widehat{\Sigma}) \oplus \sH_{-}^1(\widehat{\Sigma}),$$
where $\sH_{\pm}^1(\widehat{\Sigma}) = \{\omega\in\sH^1(\widehat{\Sigma})\,|\, \tau^\ast \omega = \pm\omega \}$. Note that the Hodge star operator $\mbox{$*$}$ of $\widehat{\Sigma}$ gives an isomorphism between $\sH_{+}^1(\widehat{\Sigma})$  and  $\sH_{-}^1(\widehat{\Sigma})$. In particular, the dimension of $\sH_{-}^1(\widehat{\Sigma}) \cap L^2_\ast(\widehat{\Sigma})$ is equal to $g + \displaystyle2\sum_{j=1}^{r} d_j+1 -1$. 

The key point is that if $\omega \in\sH_{-}^1(\widehat{\Sigma}) \cap L^2_\ast(\widehat{\Sigma})$ then $X_\omega \circ \tau = -X_\omega$, so its coordinates can be used as test functions for $Q$ as in the proof of Theorem \ref{theo:main}: suppose $\omega \in\sH_{-}^1(\widehat{\Sigma}) \cap L^2_\ast(\widehat{\Sigma})$ and $Y:\widehat{\Sigma}\rightarrow\mathbb{R}^3$ is compactly supported  with $Y\circ\tau=-Y$, we consider  the cutoff function $\varphi_R:\widehat{\Sigma}\rightarrow R$, taking $R$ big enough so that $\varphi_{R}\equiv 1$ on the support of $Y$.  Then, since $\varphi_{R}\circ\tau=\varphi_{R}$, the vector
$$X_t=\varphi_R X +tY=\varphi_R (X+tY)$$
is anti-invariant. Given this, we can argue precisely as for Theorem \ref{theo:main}.

\section{On the deformation family of the Costa surface}\label{sec:costa}

Consider $\{\Sigma_{t}\}_{t\geq 1}$ the $1$-parameter family of embedded genus one minimal surfaces with three ends, see Costa \cite{Costa:Invent}. Recall that each $\Sigma_t$ is conformally equivalent to $\CC / L(it) \setminus \{\pi(0),\pi(1/2),\allowbreak\pi(it/2)\}$, where $L(it)=\{m+int~|~m,n\in\mathbb{Z}\}$ and $\pi:\mathbb{C}\rightarrow \CC / L(it)$ is the canonical projection. We way arrange that the following isometries of $\RR^{3}$:
\[
\tau_{1}(x_{1},x_{2},x_{3}) = (-x_{1},x_{2},x_{3}), \qquad \tau_{2}(x_{1},x_{2},x_{3}) = (x_{1},-x_{2},x_{3})
\]
descend to $\overline{\Sigma_{t}}$ as reflections in $x$ and $y$ coordinates through the center of $L(it)$, where we write $z=x + i y$ for the canonical coordinate of $\CC$, see \cite{Costa:Invent, HoffmanKarcher}.

 By work of Nayatani \cite{Nayatani:CHM-index}, we know that $\Index(\Sigma_{1}) = 5$ and, by Choe \cite{Choe:vision}, that $\Index(\Sigma_{t})\geq3$ for all $t$. Suppose that $\Index(\Sigma_{t})=3$ for certain parameter $t$. We will show that this leads to a contradiction. %Recall that $\overline{\Sigma_{t}} = \CC / L(it)$ for $y=y(t)\geq 1$. 

\subsection{Decomposing the $1$-form argument with respect to the symmetries} By Proposition \ref{prop:dim}, $\sH^1(\Sigma_t) \cap L^2_\ast(\Sigma_t)$ has dimension 12. Thus, if we write $\sH$ for the subspace of such $1$-forms on $\Sigma_{t}$ that are $L^{2}_\ast$-orthogonal to $*dx^{1},*dx^{2},*dx^{3}$, then $\dim \sH =9$. 

\begin{lemm}
Suppose that 
\[
\sH = \sH^{++}\oplus \sH^{+-} \oplus \sH^{-+} \oplus \sH^{--}.
\]
is the decomposition of $\sH$ into eigenspaces of $\tau_{1},\tau_{2}$. Then:
\begin{align*}
\dim \sH^{++} & = 2\\
\dim \sH^{+-} & = 3\\
\dim \sH^{-+} & = 3\\
\dim \sH^{--} & = 1.
\end{align*}
\end{lemm}
\begin{proof}
Write $\widetilde\sH = \sH^{1}(\Sigma_{t})\cap L^{2}_{*}(\Sigma_{t})$. Because $*dx^{1},*dx^{2},*dx^{3}$ are of type $+-,-+,--$ respectively, it suffices to show that the decomposition
\[
\widetilde\sH = \widetilde\sH^{++}\oplus \widetilde\sH^{+-} \oplus \widetilde\sH^{-+} \oplus \widetilde\sH^{--}
\]
has 
\begin{align*}
\dim \widetilde\sH^{++} & = 2\\
\dim \widetilde\sH^{+-} & = 4\\
\dim \widetilde\sH^{-+} & = 4\\
\dim \widetilde\sH^{--} & = 2.
\end{align*}
Observe that the forms $dx$ and $dy$ on the torus $\overline{\Sigma_{t}}$ are of type $-+$ and $+-$, respectively. Moreover, these span $\sH^{1}(\Sigma_{t}) \cap L^{2}(\Sigma_{t})$. 

Consider $\omega \in \widetilde\sH^{++}$ (the other cases will follow similarly). Define the associated holomorphic $1$-form $\alpha = \omega + i * \omega$. We will work near the end $\pi(0)$. We have that
\[
\alpha = \alpha_{1} \frac{dz}{z} + \alpha_{2} \frac{dz}{z^{2}} + O(1)
\]
near $z=0$. Observe that  
\begin{align*}
\frac{dz}{z} & = \frac{xdx+ydy}{x^2+y^2} + i \frac{xdy-ydx}{x^2+y^2} \\
\frac{dz}{z^2}& = \frac{(x^2-y^2)dx + 2xydy}{(x^2-y^2)^2+4x^2y^2} + i \frac{(x^2-y^2)dy-2xydx}{(x^2-y^2)^2+4x^2y^2}
\end{align*}
and that the real and imaginary components of these forms obey the following parity relations:
\renewcommand{\arraystretch}{2.2}
{\begin{center}
\begin{tabular}{ |c| c | }
\hline
{ 1-form} & { type} \\\hline
$\dfrac{xdx+ydy}{x^2+y^2} $ & $++$ \\ 
$\dfrac{xdy-ydx}{x^2+y^2}$ & $--$ \\ 
$\dfrac{(x^2-y^2)dx + 2xydy}{(x^2-y^2)^2+4x^2y^2}$ & $-+$\\ 
$\dfrac{(x^2-y^2)dy-2xydx}{(x^2-y^2)^2+4x^2y^2}$ & $+-$\\ \hline
\end{tabular}
\end{center}
}
Because $\omega = \re\alpha$ was assumed to be in $\widetilde\sH^{++}$, we thus find that $\alpha_{1} \in \RR$ and $\alpha_{2} = 0$. A similar computation holds at the other ends $\pi(1/2)$ and $\pi(it/2)$. By \cite[Proposition II.5.4]{FK}, the sum of the residues of $\omega$ equals zero. Thus, because $\widetilde\sH^{++}\cap L^{2}(\Sigma_{t})=\{0\}$, we find that $\dim \widetilde\sH^{++}\leq 2$. More specifically, to each $\omega \in \widetilde\sH^{++}$ we have an injective map
\[
\omega\mapsto (\Res_{\pi(0)}\alpha,\Res_{\pi(1/2)}\alpha,\Res_{\pi(it/2)}\alpha)
\]
whose image is contained in the two-dimensional subspace $\{(u,v,w) : u+v+w=0\}\subset \RR^{3}$. This yields the asserted inequality. Arguing similarly, for the other parities, we find that the assertion in the lemma holds in the weaker form where inequalities replace the equalities.  However, since $\dim\widetilde\sH = 12$, this implies the full assertion.
 \end{proof}

Next, we also decompose the space of smooth functions as on $\Sigma_t$ as:
\begin{equation}\label{eq:decom}
C^{\infty}(\Sigma_t) = C^{++} \oplus C^{+-} \oplus C^{-+} \oplus C^{--}.
\end{equation}
Let $W$ be the span of eigenfunctions in $L^2_\ast$ with negative eigenvalue, corresponding to the Morse index of $\Sigma_t$. Recall the we are assuming $W$ has dimension 3. We decompose $W$ according to \eqref{eq:decom} as
\[
W = W^{++} \oplus W^{+-} \oplus W^{-+} \oplus W^{--}
\]
and write $\dim W^{\pm\pm} = w^{\pm\pm}$. We will show:

\begin{lemm} \label{lemm:symmetries-inequalities-CHM} The following inequalities are true:
\begin{align*}
2 = \dim \sH^{++} & \leq w^{-+} + w^{+-} + w^{++}\\
3 = \dim \sH^{+-} & \leq w^{--} + w^{++} + w^{+-}\\
3 = \dim \sH^{-+} & \leq w^{++} + w^{--} + w^{-+}\\
1 = \dim \sH^{--} & \leq w^{+-} + w^{-+} + w^{--}.\end{align*}
\end{lemm}

\begin{proof}
We will prove the first inequality and the rest will follow analogously.
Suppose $w^{-+} + w^{+-} + w^{++} \leq 1$. For $\omega \in \sH^{++}$, note that 
\[
\bangle{\omega,dx^{1}} \in C^{-+}, \bangle{\omega,dx^{2}} \in C^{+-},\bangle{\omega,dx^{3}} \in C^{++}.
\]
Thus, the requirement that $\bangle{\omega,dx^{1}}$ is orthogonal to $W$ with respect to $L^{2}_{\ast}$ amounts to at most $w^{-+}$ linear equations (pairing an element of $C^{-+}$ with e.g., $C^{++}$ vanishes automatically, by a parity argument). Thus, requiring that $X_{\omega}$ is orthogonal to $W$ takes at most $w^{-+}+w^{+-}+w^{++}$ equations. 

By our assumption, we can thus find such a $\omega \in \sH^{++}\setminus\{0\}$ that is orthogonal to $W$ in $L^2_\ast$. Repeating the proof of Theorem \ref{theo:main}, we  conclude that $\omega\in\Span\{*dx^1,*dx^2,*dx^3\}$, which is a contradiction since $\sH$ is orthogonal to $\Span\{*dx^1,*dx^2,*dx^3\}$.
\end{proof}

If we assume that $\Index(\Sigma_t) = 3$, then $\sum w^{\pm\pm} = 3$. Combined with Lemma \ref{lemm:symmetries-inequalities-CHM}, we find
\begin{equation}\label{eq:w-pmpm-bds-ind3}
w^{+-} = w^{-+} = 0,\qquad w^{++} = 2, \qquad w^{--} = 1. 
\end{equation}

We show below that this is a contradiction by comparing with Choe's vision number.\footnote{Alternatively, one may find a contradiction by first showing that $w^{--}\leq w^{+-}\leq w^{++}$ by using the variational characterization of $\omega^{\pm\pm}$ in a quarter of $\Sigma_{t}$.}

\subsection{Comparing with the vision number} 
Write $\Pi_{i} = \{\tau_{i}(x) = x\}$ for the $\{x_{i}=0\}$ planes of reflection symmetry in $\RR^{3}$. Consider $\phi_{\ell}$ the Jacobi field coming from the rotation around the $x_{3}$-axis. Write $H(\phi_{\ell}) = \{\phi_{\ell}= 0 \}$. By the symmetries, we have that
\[
\Pi_{1}\cup\Pi_{2} \subset H(\phi_{\ell}),
\]
so the vision number $v(\phi_{\ell})$ (i.e., the number of connected components of $\Sigma_{t}\setminus H(\phi_{\ell})$, see Choe \cite{Choe:vision}) satisfies $v(\phi_{\ell}) \geq 4$. See \cite[p.\ 50]{HoffmanKarcher} for a picture of one quadrant of the Costa surface. 

On the other hand, since we have assumed that $\Index(\Sigma_{t}) = 3$, we have by Choe \cite{Choe:vision}:
\[
3 = \Index(\Sigma_{t}) \geq v(\phi_{\ell}) - 1 \geq 3,
\]
we find that\footnote{In fact, one can prove this holds, even without assuming that $\Index(\Sigma_{t}) = 3$. See \cite[Remark 7.3]{HoffmanKarcher}.} $v(\phi_{\ell}) = 4$. Write $\Sigma'$ for one of the components of $\Sigma_{t}\setminus (H_{1}\cup H_{2})$. Because $v(\phi_{\ell}) = 4$, we must have that $\phi_{\ell}$ does not change sign on the interior of each connected componnet of $\Sigma'$ and vanishes at their boundary. Thus, $\Sigma'$ is a stable minimal surface with fixed boundary. By the usual argument relating odd symmetry and Dirichlet boundary conditions, this implies that $w^{--} = 0$, contradicting the previous section. 

Thus, we find that $\Index{(\Sigma_t)}\geq4$ for every $t$, completing the proof.

\bibliography{bib} 

\providecommand{\bysame}{\leavevmode\hbox to3em{\hrulefill}\thinspace}
\providecommand{\MR}{\relax\ifhmode\unskip\space\fi MR }
\providecommand{\MRhref}[2]{%
  \href{http://www.ams.org/mathscinet-getitem?mr=#1}{#2}
}
\providecommand{\href}[2]{#2}
\begin{thebibliography}{ABCS19}

\bibitem[ABCS18]{ABCS:closed-compactness}
Lucas Ambrozio, Reto Buzano, Alessandro Carlotto, and Ben Sharp,
  \emph{Geometric convergence results for closed minimal surfaces via bubbling
  analysis}, \url{https://arxiv.org/abs/1803.04956} (2018).

\bibitem[ABCS19]{ABCS:fb}
Lucas Ambrozio, Reto Buzano, Alessandro Carlotto, and Ben Sharp, \emph{Bubbling
  analysis and geometric convergence results for free boundary minimal
  surfaces}, J. \'{E}c. polytech. Math. \textbf{6} (2019), 621--664.
  \MR{4014631}

\bibitem[BS18]{buzanoSharp}
Reto Buzano and Ben Sharp, \emph{Qualitative and quantitative estimates for
  minimal hypersurfaces with bounded index and area}, Trans. Amer. Math. Soc.
  \textbf{370} (2018), no.~6, 4373--4399. \MR{3811532}

\bibitem[Cho90]{Choe:vision}
Jaigyoung Choe, \emph{Index, vision number and stability of complete minimal
  surfaces}, Arch. Rational Mech. Anal. \textbf{109} (1990), no.~3, 195--212.
  \MR{1025170 (91b:53007)}

\bibitem[CKM17]{CKM}
Otis Chodosh, Daniel Ketover, and Davi Maximo, \emph{Minimal hypersurfaces with
  bounded index}, Invent. Math. \textbf{209} (2017), no.~3, 617--664.
  \MR{3681392}

\bibitem[CM16]{CM16}
Otis Chodosh and Davi Maximo, \emph{On the topology and index of minimal
  surfaces}, Journal of Differential Geometry \textbf{104} (2016), no.~3,
  399--418.

\bibitem[Cos89]{Costa:JDG}
C.~J. Costa, \emph{Uniqueness of minimal surfaces embedded in {${\bf R}^3,$}
  with total curvature {$12\pi$}}, J. Differential Geom. \textbf{30} (1989),
  no.~3, 597--618. \MR{1021368 (90k:53011)}

\bibitem[Cos91]{Costa:Invent}
\bysame, \emph{Classification of complete minimal surfaces in {${\bf R}^3$}
  with total curvature {$12\pi$}}, Invent. Math. \textbf{105} (1991), no.~2,
  273--303. \MR{1115544 (92h:53010)}

\bibitem[EM08]{EjiriMicallef}
Norio Ejiri and Mario Micallef, \emph{Comparison between second variation of
  area and second variation of energy of a minimal surface}, Adv. Calc. Var.
  \textbf{1} (2008), no.~3, 223--239. \MR{2458236 (2009j:58019)}

\bibitem[FC85]{Fischer-Colbrie:1985}
D.~Fischer-Colbrie, \emph{On complete minimal surfaces with finite {M}orse
  index in three-manifolds}, Invent. Math. \textbf{82} (1985), no.~1, 121--132.
  \MR{808112 (87b:53090)}

\bibitem[FK80]{FK}
Hershel~M. Farkas and Irwin Kra, \emph{Riemann surfaces}, Graduate Texts in
  Mathematics, vol.~71, Springer-Verlag, New York-Berlin, 1980. \MR{583745
  (82c:30067)}

\bibitem[GL86]{Gulliver-Lawson}
Robert Gulliver and H.~Blaine Lawson, Jr., \emph{The structure of stable
  minimal hypersurfaces near a singularity}, Geometric measure theory and the
  calculus of variations ({A}rcata, {C}alif., 1984), Proc. Sympos. Pure Math.,
  vol.~44, Amer. Math. Soc., Providence, RI, 1986, pp.~213--237. \MR{840275
  (87g:53091)}

\bibitem[GNY04]{GNY:eig}
Alexander Grigor'yan, Yuri Netrusov, and Shing-Tung Yau, \emph{Eigenvalues of
  elliptic operators and geometric applications}, Surveys in differential
  geometry. {V}ol. {IX}, Surv. Differ. Geom., IX, Int. Press, Somerville, MA,
  2004, pp.~147--217. \MR{2195408 (2007f:58039)}

\bibitem[Gul86]{Gulliver:indexFTC}
Robert Gulliver, \emph{Index and total curvature of complete minimal surfaces},
  Geometric measure theory and the calculus of variations ({A}rcata, {C}alif.,
  1984), Proc. Sympos. Pure Math., vol.~44, Amer. Math. Soc., Providence, RI,
  1986, pp.~207--211. \MR{840274 (87f:53005)}

\bibitem[GY03]{GriYau03}
Alexander Grigor'yan and Shing-Tung Yau, \emph{Isoperimetric properties of
  higher eigenvalues of elliptic operators}, Amer. J. Math. \textbf{125}
  (2003), no.~4, 893--940. \MR{1993744}

\bibitem[HK97]{HoffmanKarcher}
David Hoffman and Hermann Karcher, \emph{Complete embedded minimal surfaces of
  finite total curvature}, Geometry, {V}, Encyclopaedia Math. Sci., vol.~90,
  Springer, Berlin, 1997, pp.~5--93. \MR{1490038 (98m:53012)}

\bibitem[HM90]{HoffmanMeeks:halfspace}
D.~Hoffman and W.~H. Meeks, III, \emph{The strong halfspace theorem for minimal
  surfaces}, Invent. Math. \textbf{101} (1990), no.~2, 373--377. \MR{1062966
  (92e:53010)}

\bibitem[JM83]{JorgeMeeks}
Luqu{{\'e}}sio~P. Jorge and William~H. Meeks, III, \emph{The topology of
  complete minimal surfaces of finite total {G}aussian curvature}, Topology
  \textbf{22} (1983), no.~2, 203--221. \MR{683761 (84d:53006)}

\bibitem[L{\'o}p92]{Lopez:12pi}
Francisco~J. L{\'o}pez, \emph{The classification of complete minimal surfaces
  with total curvature greater than {$-12\pi$}}, Trans. Amer. Math. Soc.
  \textbf{334} (1992), no.~1, 49--74. \MR{1058433 (93a:53008)}

\bibitem[LR89]{LopezRos:index-one}
Francisco~J. L{{\'o}}pez and Antonio Ros, \emph{Complete minimal surfaces with
  index one and stable constant mean curvature surfaces}, Comment. Math. Helv.
  \textbf{64} (1989), no.~1, 34--43. \MR{982560 (90b:53006)}

\bibitem[LR91]{LopezRos:genus-zero}
\bysame, \emph{On embedded complete minimal surfaces of genus zero}, J.
  Differential Geom. \textbf{33} (1991), no.~1, 293--300. \MR{1085145
  (91k:53019)}

\bibitem[Mor09]{Morabito}
Filippo Morabito, \emph{Index and nullity of the {G}auss map of the
  {C}osta-{H}offman-{M}eeks surfaces}, Indiana Univ. Math. J. \textbf{58}
  (2009), no.~2, 677--707. \MR{2514384 (2010j:53017)}

\bibitem[Mor17]{Morabito:comm}
\bysame, private communication, September 2017.

\bibitem[MR91]{MontielRos}
Sebasti{{\'a}}n Montiel and Antonio Ros, \emph{Schr{\"o}dinger operators
  associated to a holomorphic map}, Global differential geometry and global
  analysis ({B}erlin, 1990), Lecture Notes in Math., vol. 1481, Springer,
  Berlin, 1991, pp.~147--174. \MR{1178529 (93k:58053)}

\bibitem[Nay90]{Nayatani}
Shin Nayatani, \emph{Lower bounds for the {M}orse index of complete minimal
  surfaces in {E}uclidean {$3$}-space}, Osaka J. Math. \textbf{27} (1990),
  no.~2, 453--464. \MR{1066638 (91g:58051)}

\bibitem[Nay92]{Nayatani:CHM-index}
\bysame, \emph{Morse index of complete minimal surfaces}, The problem of
  {P}lateau, World Sci. Publ., River Edge, NJ, 1992, pp.~181--189. \MR{1209216
  (94e:58027)}

\bibitem[Nay93]{Nayatani:indexGauss}
\bysame, \emph{Morse index and {G}auss maps of complete minimal surfaces in
  {E}uclidean {$3$}-space}, Comment. Math. Helv. \textbf{68} (1993), no.~4,
  511--537. \MR{1241471}

\bibitem[Nay17]{Nayatani:comm}
\bysame, private communication, October 2017.

\bibitem[Oss64]{Osserman:FTC}
Robert Osserman, \emph{Global properties of minimal surfaces in {$E^{3}$} and
  {$E^{n}$}}, Ann. of Math. (2) \textbf{80} (1964), 340--364. \MR{0179701 (31
  \#3946)}

\bibitem[Ros92]{Ross}
Marty Ross, \emph{Complete nonorientable minimal surfaces in {$\bold R^3$}},
  Comment. Math. Helv. \textbf{67} (1992), no.~1, 64--76. \MR{1144614
  (92k:53022)}

\bibitem[Ros06]{Ros:one-sided}
Antonio Ros, \emph{One-sided complete stable minimal surfaces}, J. Differential
  Geom. \textbf{74} (2006), no.~1, 69--92. \MR{2260928 (2007g:53008)}

\bibitem[Sch83]{Schoen:symmetry}
Richard~M. Schoen, \emph{Uniqueness, symmetry, and embeddedness of minimal
  surfaces}, J. Differential Geom. \textbf{18} (1983), no.~4, 791--809 (1984).
  \MR{730928 (85f:53011)}

\bibitem[Tys87]{Tysk}
Johan Tysk, \emph{Eigenvalue estimates with applications to minimal surfaces},
  Pacific J. Math. \textbf{128} (1987), no.~2, 361--366. \MR{888524
  (88i:53102)}

\end{thebibliography}
\bibliographystyle{amsalpha}
\end{document}